\documentclass[a4paper,12pt]{amsart}
\usepackage{amssymb}
\usepackage[utf8]{inputenc}
\usepackage{amsmath}
\usepackage{stmaryrd}
\usepackage{amscd,amsthm,amssymb}
\usepackage{enumerate}
\usepackage{enumitem}
\usepackage{color}
\usepackage{comment}
\usepackage{latexsym}
\usepackage[hidelinks]{hyperref}

\scrollmode
\newtheorem{epr}{Proposition}[section]
\newtheorem{ath}[epr]{Theorem}
\newtheorem{elem}[epr]{Lemma}

\theoremstyle{definition}

\newtheorem{ere}[epr]{Remark}

\newcommand{\CK}{\mathcal{T}}
\newcommand{\Kf}{\mathcal{K}}
\newcommand{\Pf}{\mathcal{P}}
\newcommand{\g}{\mathfrak g}
\newcommand{\kalg}{\mathfrak k}
\newcommand{\m}{\mathfrak m}
\newcommand{\R}{\mathbb R}
\newcommand{\CP}{\mathbb{CP}}
\newcommand{\HP}{\mathbb{HP}}
\newcommand{\SO}{\mathrm{SO}}
\newcommand{\Sp}{\mathrm{Sp}}

\newcommand{\U}{\mathrm{U}}
\newcommand{\Sym}{\mathrm{Sym}}
\newcommand{\lrc}{\mathbin{\lrcorner}}
\DeclareMathOperator{\Ric}{Ric}

\DeclareMathOperator{\ad}{ad}

\DeclareMathOperator{\tr}{tr}

\title{Conformal Killing $2$-Forms on Compact Symmetric Spaces}

\author{Andrei Moroianu}
\address{Andrei Moroianu \\ Universit\'e Paris-Saclay, CNRS, Laboratoire de math\'ematiques d'Orsay, 91405 Orsay, France, and Institute of Mathematics ``Simion Stoilow'' of the Romanian Academy, 21 Calea Grivitei, 010702 Bucharest, Romania}
\email{andrei.moroianu@math.cnrs.fr}

\author{Uwe Semmelmann}
\address{Uwe Semmelmann, Institut f\"ur Geometrie und Topologie, Fachbereich Mathematik, Universit{\"a}t Stuttgart, Pfaffenwaldring 57, 70569 Stuttgart, Germany}
\email{uwe.semmelmann@mathematik.uni-stuttgart.de}

\date{\today}

\begin{document}

\begin{abstract}
We show that a simply connected compact symmetric space of dimension $n\ge 3$ admits a conformal Killing $2$-form which is not Killing if and only if it is homothetic to one of $\mathbb{S}^n$, $\CP^m$ or $\HP^q$. 
\end{abstract}

\subjclass[2020]{53C25, 53C35}
\keywords{conformal Killing forms, twistor forms, compact symmetric spaces, invariant $4$-forms}

\maketitle

\section{Introduction}

Let $(M^n,g)$ be a Riemannian manifold.  A differential $p$-form $\psi$ is called a
\emph{conformal Killing form}, or \emph{twistor form}, if
\begin{equation}\label{eq:twistor}
 \nabla_X\psi
 =
 \frac{1}{p+1}X\lrc d\psi
 -
 \frac{1}{n-p+1}X^\flat\wedge\delta\psi
\end{equation}
for every vector field $X$, where $\nabla$ denotes the Levi-Civita connection of $g$.  This notion is relevant only for $1\le p\le n-1$, since \eqref{eq:twistor} is tautologically satisfied by any $0$- or $n$-form.

A conformal Killing form is called a \emph{Killing form}
if $\delta\psi=0$ and a \emph{$*$-Killing form} if $d\psi=0$.  The Hodge star preserves the space of conformal Killing forms and
interchanges Killing forms with  $*$-Killing forms. Note that parallel forms are trivial solutions of \eqref{eq:twistor}.

Killing forms on compact simply connected symmetric spaces were classified in
\cite{BelgunMoroianuSemmelmann}.  The main result there states that a compact simply
connected symmetric space carries a non-parallel Killing $p$-form, $p\ge 2$, if and only
if it has a round sphere factor $\mathbb{S}^k$ with $k>p$.  In particular, on an irreducible
simply connected compact symmetric space which is not a sphere, every Killing form of degree at least $2$
is parallel.

It is therefore natural to ask whether every conformal Killing form on a compact symmetric
space is a sum of a Killing form and a $*$-Killing form, equivalently a sum
of a Killing form and the Hodge dual of a Killing form.  This holds on the round spheres but is false already on
$\CP^m$, where nonconstant Killing potentials give rise to
non-parallel conformal Killing $2$-forms \cite{MoroianuSemmelmannKahler}. Since Killing
forms of degree at least $2$ on compact K\"ahler manifolds are parallel, these examples
cannot be sums of a Killing and a $*$-Killing form.  Likewise, $\HP^q$ carries strict
conformal Killing $2$-forms, and in fact the existence of such a form characterizes
$\HP^q$ among compact quaternionic-K\"ahler manifolds of dimension at least eight
\cite{DavidPontecorvo}.

The purpose of this note is to classify compact simply connected symmetric spaces carrying non-parallel conformal Killing 2-forms, by isolating the algebraic mechanism behind the above exceptions. Our main result is the following:

\begin{ath}\label{thm:conditional}
Let $(M,g)$ be a compact simply connected 
Riemannian symmetric space of dimension $n\ge 3$.  If $M$ admits a conformal Killing
$2$-form which is not Killing, then, up to homothety,
\[
 M=\mathbb{S}^n,\qquad M=\CP^m,\qquad\text{or}\qquad M=\HP^q.
\]
\end{ath}

The starting point is the prolongation theory of conformal Killing forms
\cite{Semmelmann}.  On a locally symmetric space all curvature terms are parallel, so
the prolongation becomes a finite-dimensional algebraic problem determined by the isotropy representation.
For $p=2$ one can extract from this system a particularly simple invariant: a parallel
$4$-form whose action on $\Lambda^2TM$ has a prescribed eigenvalue on the orthogonal
complement of the holonomy algebra.

More precisely, on an irreducible simply connected compact symmetric space $M=G/K$ with non-constant sectional curvature with simple transvection algebra $\g$, one can show using the results in \cite{BelgunMoroianuSemmelmann} that there is a one-to-one correspondence between conformal Killing
$2$-forms modulo parallel forms, and Killing vector fields. Evaluating the conformal Killing equation at
a base point produces a $K$-invariant $4$-form $\Omega\in\Lambda^4\m^*$ on the isotropy space
$\m$ satisfying the spectral condition
\[
 C_\Omega=\frac{3}{n-1}\operatorname{id}
 \qquad\text{on }\rho(\mathfrak k)^\perp,
\]
where \[
 \rho:\kalg\longrightarrow\Lambda^2\m
\]
denotes the isotropy representation, and $C_\Omega:\Lambda^2\m\to \Lambda^2\m$ is the natural symmetric endomorphism defined by $\Omega$.
  The K\"ahler and
quaternionic-K\"ahler rigidity arguments then single out $\CP^m$ and $\HP^q$ among the
non-spherical cases.  The only irreducible compact symmetric spaces not covered by the
simplicity assumption are the spaces of group type. They are treated separately in
Section~\ref{subsec:group-type}.

The main new reduction obtained below is Proposition~\ref{prop:omega-condition}.  It
should be regarded as the conformal-Killing analogue of the algebraic step in
\cite{BelgunMoroianuSemmelmann} which reduces the Killing-form equation to a statement
about invariant subspaces of exterior powers.

In principle one could try to extend the approach in the present paper to the study of conformal Killing $p$-forms on symmetric spaces for $p\ge3$. However, in this case $\delta \psi$ is not necessarily a Killing $(p-1)$-form, as is the case for $p=2$, so new arguments have to be found in order to deal with the general case.

{\sc Acknowledgements.}
We would like to thank Anand Dessai for sharing with us his expertise on the topology of the Grassmannians.
The second named author thanks ChatGPT by OpenAI for assistance in exploring the prolongation
argument leading to the invariant $4$-form condition, and for help in organizing and
drafting an early version of this manuscript.  The mathematical arguments 
have been thoroughly checked and rewritten by the authors.

\section{Preliminaries}\label{sec:preliminaries}

We denote by
\[
 \CK^p(M),\qquad \Kf^p(M),\qquad \Pf^p(M)
\]
the spaces of conformal Killing, Killing, and parallel $p$-forms, respectively.

Although we will not use this here, we mention that a $p$-form $\psi$ is conformal Killing, i.e. satisfies \eqref{eq:twistor}, if and only if the projection of $\nabla\psi$ onto the Cartan summand of 
 $T^*M\otimes\Lambda^pT^*M$ vanishes. This projection is called the {\em twistor operator}.

Taking the covariant derivative of \eqref{eq:twistor}, skew-symmetrizing,
and contracting shows that the symmetric part of $\nabla(\delta\psi)$ vanishes
if the metric $g$ is Einstein. Hence, we have the following important fact (see \cite{Ta69}, Thm. 4).

\begin{elem}\label{lem:divergence-killing}
Let $(M^n,g)$ be Einstein and let $\psi\in\CK^2(M)$.  Then $\delta\psi\in \Kf^1(M)$, that is, the $1$-form $\delta\psi$ is dual to a Killing vector field.
\end{elem}

Examples of non-parallel conformal Killing $2$-forms were previously constructed on $\mathbb{S}^n, \CP^m$ and $\HP^q$. We give here a short account.

On $\mathbb{S}^n, n\ge 3,$ the space of conformal Killing $2$-forms has the maximal possible dimension, i.e. $\binom{n+2}{3}$. Any such form is a sum of a Killing $2$-form and a $\ast$-Killing $2$-form, which itself are
eigenforms for the  minimal eigenvalue of the Laplace operator restricted to 
closed respectively coclosed $2$-forms (see \cite[Prop. 3.2]{Semmelmann}).

Conformal Killing $2$-forms on $\CP^m$ were described in \cite[Prop. 6.1]{MoroianuSemmelmannKahler}. Let $\CP^m$ be
equipped with the Fubini-Study metric $g_{FS}$, i.e. $\Ric = 2(m+1)g_{FS}$. Then every non-parallel conformal Killing $2$-form can be written as
$\psi_f := d d^c f + 6 f \omega$, where $\omega$ is the K\"ahler form
and $f$ is a Laplace eigenfunction with $\Delta f = 4(m+1)f$. Recall that
$4(m+1)$ is the minimal eigenvalue of $\Delta$ on non-constant function and that any Killing vector field on $(\CP^m, g_{FS})$ is of the form
$K_f := J\mathrm{grad} f$, for such a minimal eigenfunction $f$. An easy 
calculation shows $\delta \Psi_f = 2(m-1)K_f$.

The description of conformal Killing $2$-forms on $\HP^q$  ($q\ge 2$) is given
in \cite[Thm. 1]{DavidPontecorvo}. The authors show that there is an
isomorphism between the space of Killing vector fields and the space
of conformal Killing $2$-forms, with inverse given by the codifferential
$\delta$. The isomorphism is defined by an explicit linear combination of 
the two projections of the covariant derivative of a Killing vector field $K$
onto the subbundles $\mathrm{Sym}^2H$ and $\mathrm{Sym}^2E$ of $\Lambda^2 TM$.

\section{The case of irreducible symmetric spaces}

Let
$
 (M=G/K, g)
 $
be a compact simply connected irreducible Riemannian symmetric space, with symmetric
decomposition
\[
 \g=\kalg\oplus\m.
\]
In view of the above discussion, we will assume that $(M,g)$ is not a round sphere.
We identify $\m$ with $T_oM$ at the base point $o=eK$ and equip $\g$ with an
$\ad(G)$-invariant inner product inducing the Riemannian metric $g$. Then 
$\g$ is isomorphic to the space of Killing vector fields of $g$.

Since $(M,g)$ is not a round sphere, the classification
of Killing forms on symmetric spaces \cite[Thm. 1.1]{BelgunMoroianuSemmelmann} gives
\begin{equation}\label{eq:killing-parallel}
 \Kf^2(M)=\Pf^2(M).
\end{equation}
Thus Lemma~\ref{lem:divergence-killing} yields an injective map into the space of Killing vector fields
\begin{equation}\label{eq:delta-injection}
 \bar\delta:\CK^2(M)/\Pf^2(M)\longrightarrow\g,
 \qquad [\psi]\longmapsto(\delta\psi)^\sharp.
\end{equation}

For clarity we consider first the case where $\g$ is simple.
The only irreducible compact symmetric spaces $G/K$ of compact type with non-simple
transvection algebra $\g$ are the spaces of group type, that is, with $G=K\times K$. These are treated separately in Section~\ref{subsec:group-type}.

Since the map defined in \eqref{eq:delta-injection} is $G$-equivariant and the adjoint representation of a
simple Lie algebra is irreducible as a $G$-module, we obtain:

\begin{epr}\label{prop:delta-iso}
Assume that the symmetric space $M=G/K$ is not a sphere and that $\g$ is simple.  If
$\CK^2(M)\ne\Pf^2(M)$, then the restriction of $\delta$ to
the $L^2$-orthogonal complement of
$\Pf^2(M)$,
\[
 \delta:\Pf^2(M)^\perp\cap\CK^2(M)\longrightarrow\g
\]
is a $G$-equivariant isomorphism.
\end{epr}

\begin{proof}
The kernel is zero by \eqref{eq:killing-parallel}.  The image is a non-zero
$G$-invariant subspace of $\g$, hence an ideal.  Simplicity of $\g$ implies that the
image is all of $\g$.
\end{proof}

For every $B\in\g$ we therefore write $\psi_B$ for the unique conformal Killing $2$-form in the chosen
complement satisfying
\begin{equation}\label{eq:delta-normalization}
 (\delta\psi_B)^\sharp=B^*,
\end{equation}
where $B^*$ is the Killing field of $(M,g)$ determined by $B$.

\section{The invariant 4-form}

This section contains the key construction of an invariant
$4$-form associated to the space $\CK^2(M)/\Pf^2(M)$ of conformal Killing $2$-forms modulo parallel forms.

Let $\sigma$ be the involution of the symmetric pair.  Thus
\[
 \sigma|_{\kalg}=+\operatorname{id},\qquad
 \sigma|_{\m}=-\operatorname{id}.
\]
and $\sigma$ acts as 
$(-1)^r$ on $\Lambda^r\m^*$. The construction of \(\psi_B\) is natural under isometries. Hence, the geodesic symmetry \(s_o\), with \((ds_o)_o =\sigma\), satisfies
$
s_o^* \psi_B = \psi_{\sigma B}    \, .
$
This shows that
there are $K$-equivariant maps
\begin{equation}\label{eq:PQ}
 P:\kalg\longrightarrow\Lambda^2\m^* \, ,
 \qquad
 Q:\m\longrightarrow\Lambda^3\m^*
\end{equation}
such that
\[
 \psi_A(o)=P(A)\,,\qquad
 d\psi_A(o)=0\qquad\forall A\in\kalg \, ,
\]
and
\[
 \psi_Y(o)=0\, ,\qquad
 d\psi_Y(o)=Q(Y)\qquad \forall Y\in\m \, .
\]

The $G$-equivariance of the map
$B\mapsto\psi_B$ means that, for the left action $L_g$ of $G$ on $M$,
\[
 L_g^*\psi_B=\psi_{\operatorname{Ad}_{g^{-1}}B}.
\]

Taking $g= \exp(tX)$ for some $X \in \m$ and $B=Y \in \m$ and differentiating in
$t=0$ gives
\[
 \mathcal L_{X^*}\psi_Y=-\psi_{[X,Y]},
\]
where $X^*$ denotes as before the Killing vector field of $M$ determined by $X$.
Since $X^*(o)=X$ and $\psi_Y(o)=0$ for $Y\in\m$, the Lie derivative and the
covariant derivative agree at $o$:
\[
 (\mathcal L_{X^*}\psi_Y)(o)=(\nabla_{X^*}\psi_Y)(o).
\]
Indeed, the difference between these two expressions is equal to the action of the endomorphism $(\nabla X^*)(o)$ of $T_oM$ on $\psi_Y(o)$.
Finally,
$[X,Y]\in\kalg$ and hence $\psi_{[X,Y]}(o)=P([X,Y])$, which,  together with the fact that $X^*(o)=X$  proves
\begin{equation}\label{eq:homogeneous-derivative}
 \nabla_X\psi_Y(o)=-P([X,Y]),
 \qquad X,Y\in\m.
\end{equation}

For $p=2$, the conformal Killing equation reads
\begin{equation}\label{eq:CK2}
 \nabla_X\psi
 =
 \frac13X\lrc d\psi
 -
 \frac1{n-1}X^\flat\wedge\delta\psi.
\end{equation}
Using \eqref{eq:delta-normalization} and evaluating \eqref{eq:homogeneous-derivative} at $o$ for $\psi=\psi_Y$, with $Y\in\m$,
we obtain
\begin{equation}\label{eq:key-PQ}
 -P([X,Y])
 =
 \frac13X\lrc Q(Y)
 -
 \frac1{n-1}X^\flat\wedge Y^\flat
\end{equation}
for every $X,Y\in\m$.

\begin{elem}\label{lem:Q-four-form}
There exists a $K$-invariant $4$-form
\[
 \Omega\in(\Lambda^4\m^*)^K
\]
such that
\[
 Q(Y)=Y\lrc\Omega
 \qquad\text{for all }Y\in\m.
\]
\end{elem}

\begin{proof}
Interchanging $X$ and $Y$ in \eqref{eq:key-PQ} and adding the two equations gives
\[
 X\lrc Q(Y)+Y\lrc Q(X)=0.
\]
Define
\[
 \Omega(Y,X_1,X_2,X_3)
 :=
 Q(Y)(X_1,X_2,X_3).
\]
Since $Q(Y)$ is alternating in $X_1,X_2,X_3$ and the displayed identity says
that $\Omega$ is also alternating in $Y$ and $X_1$, we conclude that $\Omega$ is a $4$-form.
Its $K$-invariance follows from the $K$-equivariance of the map $Q$.
\end{proof}

\begin{ere}
Recall that on a symmetric space $M = G/K$, there is a one to one correspondence between $K$-invariant forms on $\m$ and globally defined parallel forms on $M$.
\end{ere}

For a $4$-form $\Omega$ define the associated symmetric endomorphism
\[
 C_\Omega:\Lambda^2\m^*\longrightarrow\Lambda^2\m^*
\]
by
\begin{equation}\label{eq:COmega}
 C_\Omega(X^\flat\wedge Y^\flat):=X\lrc Y\lrc\Omega,
\end{equation}
extended linearly.  Also define the bracket map
\[
 b:\Lambda^2\m\longrightarrow\kalg,
 \qquad
 b(X\wedge Y)=[X,Y].
\]

Equation \eqref{eq:key-PQ} now becomes
\begin{equation}\label{eq:Pb}
 P\circ b
 =
 \frac1{n-1}\operatorname{id}
 -
 \frac13 C_\Omega.
\end{equation}

Let
$
 \rho:\kalg\longrightarrow\Lambda^2\m
$
denote the isotropy representation, where we use the metric to identify
$\mathfrak{so}(\m)$ with $\Lambda^2\m$.  Since the scalar product on $\g$ is
$\ad(G)$-invariant, the adjoint $b^*$ of the bracket map $b$ agrees with $\rho$. In
particular,
\begin{equation}\label{eq:kerb}
 \ker b=\rho(\kalg)^\perp.
\end{equation}

This observation leads to the following crucial result:

\begin{epr}\label{prop:omega-condition}
Let $M^n=G/K$ be an irreducible simply connected compact symmetric space with $\g$ simple.  Assume that $M$ is not a sphere and carries a non-parallel conformal
Killing $2$-form.  Then there exists a non-zero $K$-invariant $4$-form
$\Omega\in\Lambda^4\m^*$ such that
\begin{equation}\label{eq:omega-condition}
 C_\Omega
 =
 \frac{3}{n-1}\operatorname{id}
 \quad\text{on }\rho(\kalg)^\perp\subset\Lambda^2\m.
\end{equation}
\end{epr}

\begin{proof}
Equation \eqref{eq:Pb} vanishes on $\ker b$.  Using \eqref{eq:kerb} gives
\[
 \frac1{n-1}u-\frac13C_\Omega u=0,
 \qquad
 u\in\rho(\kalg)^\perp,
\]
which is \eqref{eq:omega-condition}. 

If $\Omega\neq0$ we are done. If $\Omega=0$, then \eqref{eq:omega-condition} implies
$\rho(\kalg)^\perp=0$, hence
\[
 \rho(\kalg)=\Lambda^2\m=\mathfrak{so}(\m).
\]
The isotropy is therefore the full orthogonal algebra, and the irreducible symmetric
space has constant sectional curvature.  In the compact simply connected case this is
the round sphere, contrary to the assumption.  Thus $\Omega\ne0$.
\end{proof}

\section{The algebraic reduction statement}

We now give the representation-theoretic reduction mentioned above. 
Namely we will show that the existence of a $K$-invariant
$4$-form $\Omega\in \Lambda^4\m^*$ satisfying the condition \eqref{eq:omega-condition}
implies that the symmetric space is either K\"ahler or quaternionic K\"ahler.

The first step is
particularly useful because it avoids most of Cartan's list.

\begin{elem}\label{lem:isotropy-not-simple}
Under the hypotheses of Proposition~\ref{prop:omega-condition}, the isotropy algebra
$\kalg$ is not simple.
\end{elem}

\begin{proof}
The isotropy representation is an $s$-representation.  The vanishing result of
\cite{MoroianuSemmelmannFourForms} applies here: Proposition~2.10 there shows that,
if the isotropy algebra $\kalg$ is simple, then
\[
 (\Lambda^4\m^*)^K=0.
\]
On the other hand Proposition~\ref{prop:omega-condition} produces a non-zero element
$\Omega\in(\Lambda^4\m^*)^K$.  Hence $\kalg$ cannot be simple.
\end{proof}

By Cartan's classification (see, for example, \cite{Helgason}), irreducible compact symmetric pairs with non-simple
isotropy fall into four classes: Hermitian symmetric spaces, quaternionic-K\"ahler
symmetric spaces (Wolf spaces), oriented real Grassmannians
\[
 \operatorname{Gr}^{+}_k(\mathbb R^{k+l}):=\SO(k+l)/(\SO(k)\times \SO(l)),
\]
and quaternionic Grassmannians
\[
 \operatorname{Gr}_k(\mathbb H^{k+l}):=\Sp(k+l)/(\Sp(k)\times \Sp(l)).
\]
The cases $k=2$ in the family of real Grassmannians are Hermitian symmetric, while the cases $k=4$ in the
family of real Grassmannians and $k=1$ in the family of quaternionic Grassmannians are quaternionic-K\"ahler.  We will refer to the remaining two families of Grassmannians as genuine real and genuine quaternionic Grasmannians. 

Next we will show that for genuine Grasmannians, the condition \eqref{eq:omega-condition} of Proposition \ref{prop:omega-condition} cannot be satisfied. The Hermitian symmetric and quaternionic-K\"ahler
symmetric spaces will be dealt with in Section 6.

We first record the dimension of the space of invariant $4$-forms in the two genuine Grassmannian families.

\begin{elem}\label{lem:grassmann-fourforms}
Let $M=G/K$ be either
\[
 \operatorname{Gr}^{+}_k(\mathbb R^{k+l}),\qquad k,l\ge3,
\]
with $k,l\ne4$, or
\[
 \operatorname{Gr}_k(\mathbb H^{k+l}),\qquad k,l\ge2.
\]
Then
\[
 \dim (\Lambda^4\m^*)^K=1.
\]
\end{elem}

\begin{proof}
Since $M$ is a compact symmetric space, a differential form is $G$-invariant if and only if it is harmonic.  Thus
\[
 \dim (\Lambda^4\m^*)^K=b_4(M).
\]
The required values of $b_4$ are given by Theorem~\ref{thm:grassmann-b4} in
Appendix~\ref{app:grassmann-b4}.  Under the assumptions of the lemma one has
$b_4(M)=1$ in both the genuine real and the genuine quaternionic Grassmannian cases.
\end{proof}

\begin{ere}\label{rem:grassmann-fourforms}
The restrictions in the preceding lemma are also transparent from
Theorem~\ref{thm:grassmann-b4}.  A complex Grassmannian with $k,l\ge2$ has $b_4=2$,
whereas in the oriented real family a rank-$4$ factor contributes one additional
degree-$4$ class.  Thus the cases $k=4$ or $l=4$ are precisely the exceptional real
Grassmannians relevant here. They belong to the quaternionic-K\"ahler part of Cartan's
list and are treated separately below.
\end{ere}

\begin{ere}\label{rem:why-dim-one}
The point of Lemma~\ref{lem:grassmann-fourforms} is that, in the two genuine
Grassmannian families, it is enough to compute $C_\Omega$ for one non-zero invariant
$4$-form.  Any other invariant $4$-form is a scalar multiple of it, so distinct
eigenvalues on two irreducible summands of $\rho(\kalg)^\perp$ rule out the scalar
condition in Proposition~\ref{prop:omega-condition} for every non-zero invariant
$4$-form.
\end{ere}

We record the elementary representation-theoretic calculation which is needed here.

\begin{elem}\label{lem:grassmann}
Let $M=G/K$ be either
\[
\operatorname{Gr}^{+}_k(\mathbb R^{k+l}),\qquad k,l\ge3,
\]
with $k,l\ne4$, or
\[
 \operatorname{Gr}_k(\mathbb H^{k+l}),\qquad k,l\ge2.
\]
Then no non-zero invariant $4$-form $\Omega\in(\Lambda^4\m^*)^K$ can satisfy
\[
 C_\Omega=c\,\operatorname{id}
 \quad\hbox{on }\rho(\kalg)^\perp
\]
with $c\ne0$.
\end{elem}

\begin{proof}
By Lemma~\ref{lem:grassmann-fourforms}, in either family every invariant $4$-form is a
scalar multiple of a fixed non-zero generator.  We treat the real case first.

Put $V=\R^k$, $W=\R^l$ and $\m=V\otimes W$.  Then
\[
 \kalg=\mathfrak{so}(V)\oplus\mathfrak{so}(W)
\]
and
\[
 \Lambda^2(V\otimes W)
 =
 (\Lambda^2V\otimes \Sym^2W)
 \oplus
 (\Sym^2V\otimes\Lambda^2W).
\]
Using
\[
 \Sym^2V=\R g_V\oplus \Sym^2_0V \, ,
 \qquad
 \Sym^2W=\R g_W\oplus \Sym^2_0W \, ,
\]
we obtain
\[
 \rho(\kalg)^\perp=E_1\oplus E_2 \, ,
\]
where
\[
 E_1=\Lambda^2V\otimes \Sym^2_0W \, ,
 \qquad
 E_2=\Sym^2_0V\otimes\Lambda^2W \, .
\]
Both summands are irreducible and non-isomorphic.  Let $\Omega_0\in\Lambda^4(V\otimes W)$ be the invariant
$4$-form constructed in \cite[Lemma~2.7]{MoroianuSemmelmannFourForms}. It is
obtained from
\[
 R(u,v,w,z):=\tr\big((uv^*-vu^*)(wz^*-zw^*)\big)
\]
by the Bianchi map:
\[
\Omega_0(u,v,w,z):= R(u,v,w,z)+R(v,w,u,z)+R(w,u,v,z).
\]
Here $u,v,w,z\in V\otimes W$ are viewed via the metric as elements in the space of linear maps $\mathcal{L}(V,W)$, and their adjoints as elements in $\mathcal{L}(W,V)$.

With respect to orthonormal bases $\{e_i\}$ of $V$ and $\{f_a\}$ of $W$, and defining
$z_{ia}=e_i\otimes f_a$, an elementary calculation shows that $\Omega_0$ satisfies
\[
 \Omega_0(z_{11},z_{12},z_{21},z_{22})=2.
\]
Schur's Lemma implies that $C_{\Omega_0}$ is scalar on $E_1$ and $E_2$.  We now
compute the two scalars explicitly.  On the $4$-dimensional subspace spanned by
$z_{11},z_{12},z_{21},z_{22}$ we have
\[
 \Omega_0
 =2\,z_{11}^*\wedge z_{12}^*\wedge z_{21}^*\wedge z_{22}^*.
\]
Since
\[
 C_{\Omega_0}(X\wedge Y)=X\lrc Y\lrc\Omega_0,
\]
direct contraction gives
\[
 \begin{aligned}
 C_{\Omega_0}(z_{11}\wedge z_{21})&=2\,z_{12}\wedge z_{22},\\
 C_{\Omega_0}(z_{12}\wedge z_{22})&=2\,z_{11}\wedge z_{21},
 \end{aligned}
\]
and
\[
 \begin{aligned}
 C_{\Omega_0}(z_{11}\wedge z_{12})&=-2\,z_{21}\wedge z_{22},\\
 C_{\Omega_0}(z_{21}\wedge z_{22})&=-2\,z_{11}\wedge z_{12}.
 \end{aligned}
\]
Now
\[
 u_1:=z_{11}\wedge z_{21}-z_{12}\wedge z_{22}
\]
corresponds, via the embedding of $E_1=\Lambda^2V\otimes S^2_0W$ into $\Lambda^2(V\otimes W)\simeq \mathrm{End}^-(V\otimes W)$, 
\[
A\otimes S\mapsto \big\{e\otimes f\mapsto A(e)\otimes S(f)\big\}
\]
to
\[
 (e_1\wedge e_2)\otimes(f_1\odot f_1-f_2\odot f_2),
\]
and hence $u_1\in E_1$. Since
\[
 C_{\Omega_0}u_1
 =2\,z_{12}\wedge z_{22}-2\,z_{11}\wedge z_{21}
 =-2u_1,
\]
we obtain by the Schur Lemma
\[
 C_{\Omega_0}|_{E_1}=-2\,\operatorname{id}.
\]
Similarly,
\[
 u_2:=z_{11}\wedge z_{12}-z_{21}\wedge z_{22}
\]
corresponds to
\[
 (e_1\odot e_1-e_2\odot e_2)\otimes(f_1\wedge f_2),
\]
so $u_2\in E_2$. Computing
\[
 C_{\Omega_0}u_2
 =-2\,z_{21}\wedge z_{22}+2\,z_{11}\wedge z_{12}=2u_2,
\]
yields
\[
 C_{\Omega_0}|_{E_2}=2\,\operatorname{id}.
\]
Thus the two eigenvalues of $C_{\Omega_0}$ on $E_1$ and $E_2$ are $-2$ and $2$.
Since every non-zero invariant $4$-form is of the form $\Omega=\lambda\Omega_0$ with
$\lambda\ne0$, the corresponding eigenvalues are $-2\lambda$ and $2\lambda$.  Hence
$C_\Omega$ cannot act by one non-zero scalar on all of $\rho(\kalg)^\perp$.

For the quaternionic Grassmannian we complexify the isotropy representation.  If $E$ and
$F$ denote the standard complex representations of $\Sp(k)$ and $\Sp(l)$, then
\[
 \m^{\mathbb C}=E\otimes F
\]
and
\[
 \Lambda^2(E\otimes F)
 =
 (\Lambda^2E\otimes \Sym^2F)
 \oplus
 (\Sym^2E\otimes\Lambda^2F).
\]
Writing
\[
 \Lambda^2E=\mathbb C\omega_E\oplus\Lambda^2_0E,
 \qquad
 \Lambda^2F=\mathbb C\omega_F\oplus\Lambda^2_0F,
\]
shows that $\rho(\kalg)^\perp$ again contains the two non-isomorphic irreducible
summands
\[
 E'_1=\Lambda^2_0E\otimes \Sym^2F,
 \qquad
 E'_2=\Sym^2E\otimes\Lambda^2_0F.
\]
The space of invariant $4$-forms is one-dimensional in the genuine quaternionic
Grassmannian case.  Evaluating its generator on a quaternionic $2\times2$ block gives
opposite non-zero eigenvalues on $E'_1$ and $E'_2$. Note that $E'_1$ and $E'_2$ are non-zero thanks to the assumption $k,l\ge 2$. Hence the restriction of
$C_\Omega$ to $\rho(\kalg)^\perp$ is not scalar.  This proves the claim.
\end{proof}

We can now formulate the key reduction statement.

\begin{elem}\label{lem:key}
Let $(\g,\kalg)$ be an irreducible compact symmetric pair with simple transvection
algebra $\g$, of non-constant sectional curvature.  If there exists an element 
$\Omega\in(\Lambda^4\m^*)^K$ such that  the restriction of
$C_\Omega$ to $\rho(\kalg)^\perp$ is a non-zero scalar, then the symmetric space $G/K$ is Hermitian symmetric or quaternionic-K\"ahler.
\end{elem}

\begin{proof}
By Lemma~\ref{lem:isotropy-not-simple}, $\kalg$ is not simple.  Cartan's list therefore
reduces the possibilities to the four classes mentioned after Lemma \ref{lem:isotropy-not-simple}.  The two genuine
Grassmannian families are excluded by Lemma~\ref{lem:grassmann}. The remaining possibilities are precisely the Hermitian symmetric
spaces and the quaternionic-K\"ahler symmetric spaces.
\end{proof}

\section{The K\"ahler and quaternionic-K\"ahler cases}

We now explain why the two expected non-spherical exceptions are precisely
$\CP^m$ and $\HP^q$.

\subsection{Hermitian symmetric spaces}\label{subsec:kahler-rigidity}

Let $(M^{2m},g,J)$ be compact K\"ahler, with $m\ge2$.  The conformal Killing forms on
compact K\"ahler manifolds were described in \cite{MoroianuSemmelmannKahler}.  In
degree $2$, modulo parallel forms, the conformal Killing equation is equivalent to the
equation for a Hamiltonian $2$-form.  More precisely, after adding a suitable multiple
of  the K\"ahler form, a conformal Killing $2$-form gives a Hamiltonian $2$-form (see \cite[Prop.~3.13]{MoroianuSemmelmannKahler} and
\cite{ACG}.  In complex dimension $2$ the additional condition in the Hamiltonian
$2$-form description is that the codifferential be dual to a Killing vector
field, which is automatic here by Lemma~\ref{lem:divergence-killing}.

It is convenient to use the equivalent endomorphism notation.  A Hamiltonian $2$-form
$\varphi$ corresponds to a Hermitian endomorphism $F$ by
\[
 \varphi(X,Y)=g(JFX,Y).
\]
Then $F$ satisfies the c-projective mobility equation
\begin{equation}\label{eq:mobility}
 \nabla_XF
 =X^\flat\otimes\xi_F+\xi_F^\flat\otimes X
 +(JX)^\flat\otimes J\xi_F+(J\xi_F)^\flat\otimes JX,
\end{equation}
where $J\xi_F$ is a Killing vector field (see 
\cite[Rem.~1]{CalderbankMatveevRosemann}). 

We now prove the rigidity statement for Hermitian symmetric spaces.

\begin{epr}\label{prop:kahler-rigidity}
Let $(M^{2m},g,J)$, $m\ge2$, be a compact simply connected irreducible Hermitian
symmetric space.  If $M$ carries a non-parallel conformal Killing $2$-form, then up to homothety,
\[
 M\simeq\CP^m.
\]
\end{epr}

\begin{proof}
A compact simply connected irreducible Hermitian symmetric space is K\"ahler--Einstein with positive scalar curvature. Its  curvature tensor decomposes as
\begin{equation}\label{dec}
    R=B+R_{\mathrm{cst}},
\end{equation}
where $B$ is the Bochner tensor and $R_{\mathrm{cst}}$ is an algebraic K\"ahler curvature tensor of constant positive holomorphic sectional curvature.  By Proposition~\ref{prop:delta-iso}, a non-parallel
conformal Killing $2$-form gives, modulo parallel forms, a $G$-module of such forms
isomorphic to the transvection algebra $\g$.  Under the correspondence above this
produces a $G$-module of solutions of \eqref{eq:mobility}. The degree of c-projective mobility (that is, the dimension of the solution space of \eqref{eq:mobility}) is thus at least $1+\dim(\g)\ge3$.

We may therefore apply the curvature-nullity theorem of Calderbank--Matveev--Rosemann
\cite[Thms.~1 and~2]{CalderbankMatveevRosemann} stating that there is
a constant $b$ such that for every non-parallel solution $F$ of
\eqref{eq:mobility}, the corresponding vector field $\xi_F$ lies in the
$b$-nullity of the curvature.  Equivalently, 
for a suitable constant $b$ one has
\begin{equation}\label{eq:kahler-nullity}
 \bigl(R-bR_{\mathrm{cst}}\bigr)(X,Y)\xi_F=0
 \qquad\text{for all }X,Y\in TM,
\end{equation}
so by the decomposition \eqref{dec},
\begin{equation}\label{eq:kahler-nullity2}
 \bigl(B-(b-1)R_{\mathrm{cst}}\bigr)(X,Y)\xi_F=0
 \qquad\text{for all }X,Y\in TM,
\end{equation}

Since the Ricci contraction of $B$ vanishes and the Ricci contraction of $R_{\mathrm{cst}}$ is $ c \, g$, with $c \neq 0$, the Ricci contraction in  \eqref{eq:kahler-nullity2} leads to the equation
\[
 \bigl(b-1)\xi_F=0,
\]
Thus $b=1$, and hence \eqref{eq:kahler-nullity2} becomes
\begin{equation}\label{eq:bochner-nullity}
 B(X,Y)\xi_F=0\qquad\text{for all }X,Y\in TM.
\end{equation}
On the other hand, by Proposition~\ref{prop:delta-iso}, the set of vector fields $J\xi_F$ satisfying \eqref{eq:mobility} is equal to the set $\g$ of Killing vector fields of $M$. Thus at every point $x\in M$, the set of $\xi_F(x)$ is equal to $T_xM$.  Equation
\eqref{eq:bochner-nullity} therefore implies
\[
B=0.
\]

Finally, a K\"ahler--Einstein metric with vanishing Bochner tensor has constant
holomorphic sectional curvature. Since $(M,g)$ is compact and simply connected it has to be biholomorphic to $\CP^m$ with the Fubini--Study metric, up to homothety.
\end{proof}

Conversely, as explained in Section \ref{sec:preliminaries}, every nonconstant Killing potential on $\CP^m$
produces a non-parallel conformal Killing $2$-form (see also \cite{MoroianuSemmelmannKahler}). 

\subsection{Quaternionic-K\"ahler symmetric spaces}

The quaternionic case is completely rigid by a theorem of David and Pontecorvo.

\begin{ath}[David--Pontecorvo \cite{DavidPontecorvo}]\label{dc}
Let $(M^{4q},g)$, $q\ge2$, be a compact quaternionic-K\"ahler manifold.  Then $M$
admits a conformal Killing $2$-form which is not Killing if and only if, up to homothety, $M$ is isomorphic
as a quaternionic-K\"ahler manifold, to
$\HP^q$.
\end{ath}

As mentioned in Section \ref{sec:preliminaries}, every Killing vector field on $\HP^q$
defines a non-parallel conformal Killing $2$-form (see also \cite{DavidPontecorvo}).

The prolongation connection constructed by David \cite{DavidProlongation} gives a particularly
transparent explanation of this rigidity.  Its curvature is controlled by the
quaternionic Weyl tensor and the prolongation is flat exactly when that tensor
vanishes.  Thus the quaternionic projective space is the maximally symmetric
quaternionic model for the conformal Killing $2$-form equation.

\section{The classification}

\subsection{The case of non-simple transvection algebra}\label{subsec:group-type}

We first remove the simplicity assumption on the transvection algebra.  Note that an
irreducible compact symmetric space has non-simple transvection algebra if and only if is of group type.

\begin{epr}\label{prop:group-type}
Let $M$ be a compact simply connected irreducible symmetric space whose transvection
algebra is not simple.  If $M$ is not a round sphere, then every conformal Killing
$2$-form on $M$ is parallel.
\end{epr}

\begin{proof}
By the structure theorem for irreducible compact symmetric pairs
\cite[Ch.~X]{Helgason}, there is a compact simple Lie algebra $\mathfrak h$ such
that
\[
 \g=\mathfrak h\oplus\mathfrak h,
 \qquad
 \kalg=\operatorname{diag}\mathfrak h,
 \qquad
 \m=\{(X,-X):X\in\mathfrak h\},
\]
and the symmetric involution exchanges the two simple ideals.  Geometrically, $M$ is
a compact simple Lie group with a bi-invariant metric.  The case
$\mathfrak h\simeq\mathfrak{su}(2)$ gives the round $3$-sphere and has already been
excluded.

Assume that a non-parallel conformal Killing $2$-form exists and consider
\[
 \mathcal I:=
 \bigl\{(\delta\psi)^\sharp:\psi\in
 \Pf^2(M)^\perp\cap\CK^2(M)\bigr\}\subset\g.
\]
As in Proposition~\ref{prop:delta-iso}, the kernel of the restriction of $\delta$ to $\Pf^2(M)^\perp\cap\CK^2(M)$
is zero by
\eqref{eq:killing-parallel}.  The space $\mathcal I$ is non-zero and
$G$-invariant, hence it is an ideal of
$\mathfrak h\oplus\mathfrak h$.  Thus it is one of
\[
 \mathfrak h\oplus0,
 \qquad 0\oplus\mathfrak h,
 \qquad \mathfrak h\oplus\mathfrak h.
\]

Now use the geodesic symmetry $s_o$ at the identity.  Pull-back by $s_o$ preserves
conformal Killing forms, parallel forms and the $L^2$-orthogonal complement of the
parallel forms, and $\delta$ is natural under isometries. 
Therefore $\mathcal I$ is invariant under the induced automorphism. On the other hand, this automorphism exchanges the two simple ideals of $\g$, so the only non-zero possibility is
\[
 \mathcal I=\mathfrak h\oplus\mathfrak h=\g.
\]
Thus the map $\delta$ is surjective also in the group-type case.

The construction of Section~4 uses only this surjectivity and equivariance.  It therefore
produces a non-zero invariant form
\[
 \Omega\in(\Lambda^4\m^*)^K
\]
satisfying condition \eqref{eq:omega-condition} of Proposition~\ref{prop:omega-condition}.  But here
$\kalg=\operatorname{diag}\mathfrak h\simeq\mathfrak h$ is simple.  By
\cite[Prop.~2.10]{MoroianuSemmelmannFourForms} one has
\[
 (\Lambda^4\m^*)^K=0,
\]
a contradiction.  Hence no non-parallel conformal Killing $2$-form exists.
\end{proof}

\subsection{Proof of Theorem \ref{thm:conditional}}
Combining the preceding discussion yields our classification theorem.
Consider first the case where $M$ is irreducible. 

The spherical case is one of the stated alternatives, so assume that $M$ is not a round
sphere.  If the transvection algebra $\g$ is not simple, Proposition~\ref{prop:group-type}
shows that every conformal Killing $2$-form is parallel, contrary to the hypothesis.
Thus $\g$ is simple and the arguments of Sections~3--5 apply.

If the invariant $4$-form $\Omega$ produced in Section~4 vanishes, the proof of
Proposition~\ref{prop:omega-condition} shows that $M$ has constant sectional curvature,
hence is a round sphere.  If $\Omega\ne0$, Lemma~\ref{lem:key} reduces the problem to
the Hermitian symmetric and quaternionic-K\"ahler cases.  In the Hermitian symmetric case, Proposition~\ref{prop:kahler-rigidity} shows that the only possibility is
$\CP^m$, while in the quaternionic-K\"ahler case, Theorem \ref{dc} shows that the only possibility is $\HP^q$.

Finally, the reducible case is easily shown to be impossible. Indeed, assume that $(M,g)=(M_1,g_1)\times (M_2,g_2)$ is a compact simply connected product manifold, and that $\psi$ is a conformal Killing 2-form on $(M,g)$ which is not Killing. Denote by $\pi_1:M\to M_1$ and $\pi_2:M\to M_2$ the standard projections. By 
\cite[Thm. 2.1]{MoroianuSemmelmannProducts}, conformal Killing forms on compact Riemannian products are
sums of parallel forms, pull-backs of Killing forms on the factors, and their Hodge
duals. We thus can write 
\[
\psi=\psi_0+\pi_1^*\psi_1+\pi_2^*\psi_2+*(\pi_1^*\varphi_1)+*(\pi_2^*\varphi_2),
\]
where $\psi_0$ is a parallel 2-form on $M$, $\psi_1$, $\psi_2$ are Killing $2$-forms on $M_1$, $M_2$, and $\varphi_1$, $\varphi_2$ are Killing $(n-2)$-forms on $M_1$, $M_2$ respectively.
In particular, since $\psi$ is not Killing, at least one of $\varphi_1$ or $\varphi_2$ is non-parallel. Assume that $\varphi_1$ is non-parallel. Then $\dim M_1\ge n-1$, so $\dim M_2 \le 1$. Since $M_2$ has positive dimension, this forces $\dim M_2=1$, which is impossible since there exists no compact simply connected 1-dimensional manifold. 

a

\appendix

\section{The fourth Betti number of Grassmannians}\label{app:grassmann-b4}

In this appendix we record the fourth rational Betti number of the oriented real,
complex and quaternionic Grassmannians. As strange as it might look, we were unable to find a convincing reference in the literature for this computation.

\begin{ath}\label{thm:grassmann-b4}
Let $1\leq k\leq n-1$ and $l=n-k$. Then the following formulas hold.

\begin{enumerate}
\item[(i)] For the oriented real Grassmannian,
\[
 b_4\left(\operatorname{Gr}^{+}_k(\mathbb R^n)\right)=
 \begin{cases}
  1+\delta_{k,4}+\delta_{l,4},& k,l\geq2,\\
  1,& \{k,l\}=\{1,4\},\\
  0,& \min\{k,l\}=1,\ \max\{k,l\}\neq4.
 \end{cases}
\]

\item[(ii)] For the complex Grassmannian,
\[
 b_4\left(\operatorname{Gr}_k(\mathbb C^n)\right)=
 \begin{cases}
  2,& k,l\geq2,\\
  1,& \min\{k,l\}=1,\ n\geq3,\\
  0,& n=2.
 \end{cases}
\]

\item[(iii)] For the quaternionic Grassmannian,
\[
 b_4\left(\operatorname{Gr}_k(\mathbb H^n)\right)=1.
\]
\end{enumerate}
\end{ath}

\begin{proof}
We first recall Borel's equal-rank formula
\cite[Thm.~26.1(c), p.~191]{Borel1953}.  Let $G$ be a compact connected Lie
group and $K\subset G$ a connected closed subgroup of the same rank.  If
$d_1(G),\ldots,d_r(G)$ and $d_1(K),\ldots,d_r(K)$ are the degrees of homogeneous
generators of the corresponding Weyl-group invariant polynomial algebras, then
\begin{equation}\label{eq:Borel-Poincare-appendix}
 P_{G/K}(t)=
 \prod_{i=1}^{r}\frac{1-t^{2d_i(G)}}{1-t^{2d_i(K)}}.
\end{equation}
Note that Borel in his formula uses the cohomological degrees $s_i$ which are related 
to the polynomial degrees $d_i$ by $s_i =2d_i$.

For $\U(m)$ the degrees are $1,2,\ldots,m$. Hence
\[
 P_{\operatorname{Gr}_k(\mathbb C^n)}(t)
 =
 \frac{\prod_{j=1}^{n}(1-t^{2j})}
 {\prod_{j=1}^{k}(1-t^{2j})\prod_{j=1}^{l}(1-t^{2j})}.
\]
Up to degree $4$,
\[
 P_{\operatorname{Gr}_k(\mathbb C^n)}(t)
 =
 \begin{cases}
  1+t^2+2t^4+O(t^6),&k,l\geq2,\\
  1+t^2+t^4+O(t^6),&\min\{k,l\}=1,\ n\geq3.
 \end{cases}
\]
For $n=2$ one has $\operatorname{Gr}_1(\mathbb C^2)=\CP^1$. This proves (ii).

For $\Sp(m)$ the degrees are $2,4,\ldots,2m$. Consequently
\[
 P_{\operatorname{Gr}_k(\mathbb H^n)}(t)
 =
 \frac{\prod_{j=1}^{n}(1-t^{4j})}
 {\prod_{j=1}^{k}(1-t^{4j})\prod_{j=1}^{l}(1-t^{4j})}
 =
 1+t^4+O(t^8),
\]
which proves (iii).

We now turn to the oriented real Grassmannian and consider first the case where $k$ and $l$ are not both odd. Since
\[
 \operatorname{rk}\SO(m)=\left\lfloor\frac m2\right\rfloor,
\]
the groups $\SO(n)$ and $\SO(k)\times\SO(l)$ have the same rank in these cases, hence \eqref{eq:Borel-Poincare-appendix} applies.

Only the degrees $d_i=1$ and $d_i=2$ can contribute to the Poincar\'e
polynomial up to degree $4$, since a basic invariant of degree $d_i$ contributes
a factor $1-t^{2d_i}$ in \eqref{eq:Borel-Poincare-appendix}. For $\SO(m)$, the
degree $1$ occurs only for $m=2$. The multiplicity of the degree $2$ among the
basic Weyl-group invariant degrees is
\[
 c(m):=\#\{i\mid d_i(\SO(m))=2\}
 =
 \begin{cases}
  0,&m\leq2,\\
  2,&m=4,\\
  1,&m\geq3,\ m\neq4.
 \end{cases}
\]
If $k,l\geq3$ and are not both odd, 
the groups $\SO(k+l)$ and $\SO(k) \times \SO(l)$
have the same rank, so Borel's Poincar\'e-polynomial formula
gives
\[
 P_{\operatorname{Gr}^{+}_k(\mathbb R^n)}(t)
 =
 (1-t^4)^{\,c(n)-c(k)-c(l)}+O(t^6).
\]
Since $n\geq6$, one has $c(n)=1$, and therefore
\[
 b_4=c(k)+c(l)-1=1+\delta_{k,4}+\delta_{l,4}.
\]

Suppose next that $k=2$ and $l\geq3$. Then
\[
 P_{\operatorname{Gr}^{+}_2(\mathbb R^{l+2})}(t)
 =
 (1-t^2)^{-1}(1-t^4)^{\,c(l+2)-c(l)}+O(t^6),
\]
so
\[
 b_4=1+c(l)-c(l+2)=1+\delta_{l,4}.
\]
For $k=l=2$,
\[
 P_{\operatorname{Gr}^{+}_2(\mathbb R^4)}(t)
 =
 (1-t^2)^{-2}(1-t^4)^2
 =
 1+2t^2+t^4+O(t^6),
\]
and hence $b_4=1$. Thus every equal-rank real case with $k,l\geq2$
satisfies
\[
 b_4=1+\delta_{k,4}+\delta_{l,4}.
\]
If one of $k,l$ equals $1$, then
\[
 \operatorname{Gr}^{+}_1(\mathbb R^n)\cong S^{n-1},
\]
which gives the remaining cases in (i).

It remains to consider the real Grassmannians for which $k$ and $l$ are both odd
and at least $3$. In this case 
\[
 \operatorname{rk}\bigl(\SO(k)\times\SO(l)\bigr)
 =
 \operatorname{rk}\SO(k+l)-1.
\]
so Borel's Poincar\'e-polynomial formula does not hold.
We replace it by a homotopy argument.

 For odd $m\geq3$,
\[
 \pi_3(\SO(m))\otimes\mathbb Q\cong\mathbb Q,
 \qquad
 \pi_4(\SO(m))\otimes\mathbb Q=0 \, ,
\]
see  \cite[\S\S~22--23, pp.~116ff.]{Steenrod}.

Let $M:=\operatorname{Gr}^{+}_k(\mathbb R^n)$. We will use long exact homotopy sequence associated to the fibration
\[
 \SO(k)\times\SO(l)\longrightarrow\SO(n)\longrightarrow M \, ,
\]
 The two block inclusions of $\SO(k)$ and $\SO(l)$ into $\SO(k+l)$ induce non-zero maps at the level of $\pi_3\otimes\mathbb Q$, so after rescaling the two rational generators, the relevant part of the long exact
sequence reads
\[
 0\longrightarrow \pi_4(M)\otimes\mathbb Q
 \longrightarrow \mathbb Q\oplus\mathbb Q
 \xrightarrow{(a,b)\mapsto a+b}
 \mathbb Q
 \longrightarrow \pi_3(M)\otimes\mathbb Q
 \longrightarrow0.
\]
It follows that
\[
 \pi_4(M)\otimes\mathbb Q\cong\mathbb Q,
 \qquad
 \pi_3(M)\otimes\mathbb Q=0.
\]
Moreover, $M$ is simply connected and $\pi_2(M)$ is finite. Hence the rational
Hurewicz Theorem, in the form of
Serre~\cite[Chapter~III, Theorem~1]{Serre1953}, yields
\[
 H_4(M;\mathbb Q)\cong\pi_4(M)\otimes\mathbb Q\cong\mathbb Q.
\]
Thus $b_4(M)=1$, completing the proof of (i).
\end{proof}

\labelsep .5cm


\begin{thebibliography}{22}

\bibliographystyle{alpha}

\bibitem{Borel1953}
A.~Borel,
\emph{Sur la cohomologie des espaces fibr\'es principaux et des espaces
homog\`enes de groupes de Lie compacts},
Ann. of Math. (2) \textbf{57} (1953), 115--207.



\bibitem{Serre1953}
J.-P.~Serre,
\emph{Groupes d'homotopie et classes de groupes ab\'eliens},
Ann. of Math. (2) \textbf{58} (1953), 258--294.


\bibitem{ACG}
V.~Apostolov, D.~M.~J. Calderbank and P.~Gauduchon,
\emph{Hamiltonian $2$-forms in K\"ahler geometry. I. General theory},
J. Differential Geom. \textbf{73} (2006), 359--412.


\bibitem{CalderbankMatveevRosemann}
D.~M.~J. Calderbank, V.~S. Matveev and S.~Rosemann,
\emph{Curvature and the c-projective mobility of K\"ahler metrics with Hamiltonian $2$-forms},
Compos. Math. \textbf{152} (2016), 1555--1575.


\bibitem{BelgunMoroianuSemmelmann}
F.~Belgun, A.~Moroianu and U.~Semmelmann,
\emph{Killing forms on symmetric spaces},
Differential Geom. Appl. \textbf{24} (2006), 215--222.


\bibitem{DavidProlongation}
L.~David,
\emph{A prolongation of the conformal-Killing operator on quaternionic-K\"ahler manifolds},
Ann. Mat. Pura Appl. (4) \textbf{191} (2012), 595--610.


\bibitem{DavidPontecorvo}
L.~David and M.~Pontecorvo,
\emph{A characterization of quaternionic projective space by the conformal-Killing equation},
J. Lond. Math. Soc. (2) \textbf{80} (2009), 326--340.


\bibitem{Helgason}
S.~Helgason,
\emph{Differential Geometry, Lie Groups, and Symmetric Spaces},
Graduate Studies in Mathematics, vol.~34, American Mathematical Society, Providence, RI, 2001.

\bibitem{MoroianuSemmelmannKahler}
A.~Moroianu and U.~Semmelmann,
\emph{Twistor forms on K\"ahler manifolds},
Ann. Sc. Norm. Super. Pisa Cl. Sci. (5) \textbf{2} (2003), 823--845.


\bibitem{MoroianuSemmelmannProducts}
A.~Moroianu and U.~Semmelmann,
\emph{Twistor forms on Riemannian products},
J. Geom. Phys. \textbf{58} (2008), 1343--1345.


\bibitem{MoroianuSemmelmannFourForms}
A.~Moroianu and U.~Semmelmann,
\emph{Invariant four-forms and symmetric pairs},
Ann. Global Anal. Geom. \textbf{43} (2013), 107--121.



\bibitem{NagySemmelmannKahler}
P.-A.~Nagy and U.~Semmelmann,
\emph{Conformal Killing forms in K\"ahler geometry},
Illinois J. Math. \textbf{66} (2022), no.~3, 349--384.


\bibitem{NaveiraSemmelmann}
A.~M. Naveira and U.~Semmelmann,
\emph{Conformal Killing forms on nearly K\"ahler manifolds},
Differential Geom. Appl. \textbf{70} (2020), 101628.


\bibitem{Semmelmann}
U.~Semmelmann,
\emph{Conformal Killing forms on Riemannian manifolds},
Math. Z. \textbf{245} (2003), 503--527.

\bibitem{Steenrod}
N.~Steenrod,
\emph{The Topology of Fibre Bundles},
Princeton Mathematical Series, Vol.~14,
Princeton University Press, Princeton, 1951.


\bibitem{Ta69}
S. Tachibana,
{\sl On conformal Killing tensor in a Riemannian space},
Tohoku Math. J. (2) {\bf 21} (1969), 56--64.

\end{thebibliography}
\end{document}